\documentclass[11pt]{amsart}
\usepackage{xcolor}
\usepackage{cancel}
\usepackage{setspace}

\usepackage{lscape}

\newcommand{\F}{\mathcal{F}}
\newcommand{\I}{\mathcal {I}}

\newcommand{\e}{\varepsilon}

\newcommand{\ol}{\overline}

\newcommand{\N}{\mathbb{N}}
\newcommand{\K}{\mathcal{K}}
\newcounter{cnt1}
\newcounter{cnt2}
\newcommand{\blr}{\begin{list}{$(\roman{cnt1})$}
 {\usecounter{cnt1} \setlength{\topsep}{0pt}
 \setlength{\itemsep}{0pt}}}
\newcommand{\bla}{\begin{list}{$($\alph{cnt2}$)$}
 {\usecounter{cnt2} \setlength{\topsep}{0pt}
 \setlength{\itemsep}{0pt}}}
\newcommand{\el}{\end{list}}
\newtheorem{thm}{Theorem}
\newtheorem{lem}[thm]{Lemma}
\newtheorem{cor}[thm]{Corollary}
\newtheorem{ex}[thm]{Example}

\newtheorem{Def}[thm]{Definition}

\newtheorem{prop}[thm]{Proposition}
\newtheorem{rem}[thm]{Remark}
\newcommand{\Rem}{\begin{rem} \rm}
\newcommand{\bdfn}{\begin{Def} \rm}
\newcommand{\edfn}{\end{Def}}

\begin{document}
\large
\title[Geometry]{Geometric conditions which imply $w^*$-normal structure}
\author[Gothwal]{Deepak Gothwal}
\address[Deepak Gothwal]
{Department of Mathematics\\
Indian Institute of Technology, Kharagpur\\
West Bengal-721302\\ India,
\textit{E-mail~:}
\textit{deepakgothwal190496@gmail.com}}
\author[Rao]{T. S. S. R. K. Rao}
\address[T. S. S. R. K. Rao]
{Department of Mathematics\\
Shiv Nadar Institution of Eminence \\
Gautam Buddha Nagar\\ UP-201314 \\ India,
\textit{E-mail~:}
\textit{srin@fulbrightmail.org}}
\subjclass[2020]{Primary 46B20, 46B22, 46B25, 46B50, 46E15, 46L05, 46L10, 46A55 47H10, 47L05, 47L15}
\keywords{Lim's condition, $w^*$-normal structure, convexity, fixed point theorems, $C^*$-algebras, $L^1$-predual space, uniform algebras.}
\begin{abstract}
 Let $X$ be a complex Banach space. In their seminal work, Lau and Mah (\cite{LM}) study a geometric property called Lim's condition which, for separable $X$, implies that for any non-empty $w^*$-compact convex set $K \subset X^*$, any nonexpansive mapping $T: K \rightarrow K$ has a fixed point.
 In this paper, we study a weaker form of Lim's condition, which we call weak-Lim's condition. We prove that if $\ell_1 \not\hookrightarrow X$, then weak-Lim's condition on $X^*$ implies the fixed point property for non-expansive maps on weakly compact subsets of $X$. In the case of a commutative $C^\ast$-algebra, we prove that weak-Lim's condition is equivalent to Lim's condition. We show that Lim's condition characterizes $C^*$-algebras which are a $c_0$-direct sum of finite-dimensional operator spaces, while weak-Lim's condition characterizes $C^*$-algebras which are a $c_0$-direct sum of the spaces of compact operators on Hilbert spaces. For a uniform algebra, we extend a result of Lau and Mah to show that weak-Lim's condition implies that the space is finite-dimensional. In the case of an $L^1$-predual space, we show that this condition implies $k$-smoothness of the norm in the sense considered in \cite{LR}.
\end{abstract}

\maketitle
\section{Introduction} 
Throughout this paper, $X$ denotes a complex Banach space and $X^*$ denotes the dual of $X$. $B(X)$ and $S(X)$ are the notations used for the closed unit ball and the unit sphere of $X$, respectively. A central theme in Banach space theory is the study of geometric conditions that guarantee a ($w^*$-)normal structure, which implies the existence of fixed points for non-expansive nonlinear maps on ($w^*$-)weakly compact sets. An interesting property in this context, which arises from the study of $w^*$-normal structure and the fixed-point property in operator spaces, is Lim's condition. (See \cite{Lim} for the undefined notions in this context.)

\begin{Def}
  $X$ is said to satisfy Lim's condition if for every bounded net $\{\phi_{\alpha}\}$ in $X^*$ with $\|\phi_{\alpha}\|=s$ for some $s \geq 0$ and $\phi_{\alpha} \xrightarrow{w^*} 0$, we have that for any $\phi \in X^*$,
\[
\lim_{\alpha}\|\phi_{\alpha}+\phi\|=s + \|\phi\|.
\]  
\end{Def}
Note that Lau and Mah referred to this as Lim's condition on $X^*$ while we call it Lim's condition on $X$.

This property originates in the work of Lim in \cite[Theorem 2]{Lim}: \emph{If $X$ is separable and satisfies Lim's condition, then every $w^*$-closed convex subset $K$ of $X^*$ has $w^*$-normal structure, that is, for any non-trivial convex set $H \subset K$, there is $x_0 \in H$ such that $\sup\{\|x_0-y\|:y\in H\}< \sup\{\|x-y\|:x~,y \in H\}$}. Thus, by \cite[Theorem 1]{Lim}, every non-expansive mapping  $T: K \rightarrow K$ has a fixed point.

These concepts were later studied by Lau and Mah in \cite{LM} as a sufficient condition for $w^*$-normal structure in the spaces of operators on a Hilbert space (\cite[Lemma 4]{LM}).

We state their lemma as a theorem below.
\begin{thm} \label{thm0}
If $X$ has Lim's condition, then
\bla
\item $w^*$ and norm topologies coincide on $S(X^*)$ and

\item for any $0 <\e <2$, whenever $\{\phi_{\alpha}\}$ is a net in $B(X^*)$ with $\|\phi_{\alpha}-\phi_{\beta}\| \geq \e$ for all $\alpha \neq \beta$ and $\phi \in B(X^*)$ is such that $\phi_{\alpha} \xrightarrow{w^\ast} \phi$, then
\[
\|\phi\| \leq 1 - \e/2.
\]
\el
\end{thm}

In the geometry of Banach spaces, it is often observed that restricting a property to norm-attaining functionals is sufficient to obtain desirable results, as it contains the extremal information of the space. This motivates us to analyze if the same holds for the existence of normal structure as far as Lim's condition is concerned. So, we first consider the following weaker form of Lim's condition.

\begin{Def}
$X$ is said to have weak-Lim's condition if for any $0 <\e <2$, whenever $\{\phi_{\alpha}\}$ is a net of norm-attaining functionals in $B(X^*)$ with $\|\phi_{\alpha}-\phi_{\beta}\| \geq \e$ for all $\alpha \neq \beta$ and $\phi \in B(X^*)$ is a norm attaining functional such that $\|\phi_\alpha\| \rightarrow \|\phi\|$ and $\phi_{\alpha} \xrightarrow{w^\ast} \phi$, then
\[
\|\phi\| \leq 1 - \e/2.
\]
\end{Def}

One of our major results in this paper focuses on proving that if the sequence space $\ell_1 \not\hookrightarrow X$ and $X^*$ has weak-Lim's condition, then every weakly compact subset of $X$ has a normal structure, which implies the existence of a fixed point for non-expansive maps on weakly compact sets (Corollary ~\ref{cornormal}).

Note that the above definition weakens condition $(b)$ of Theorem ~\ref{thm0}. Our first example is the commutative version of one of our main results.

\begin{ex} \label{excomalg}
    
Let  $\Omega$ be an infinite locally compact Hausdorff space. We denote by $C_0(\Omega)$ the space of functions vanishing at infinity equipped with the supremum norm. Suppose $C_0(\Omega)$ has weak-Lim's condition. Let $\omega \in \Omega$ be an accumulation point. We may assume without loss of generality that there is a net $\{\omega_{\alpha}\}$ of distinct elements such that $ \omega _ {\alpha} \rightarrow \omega$. Note that the Dirac measures $\delta(\omega_{\alpha})$ and $\delta(\omega)$ are norm attaining functionals in $S(C_0(\Omega)^*)$. Also note that for $ \omega_1\neq \omega_2$, $\|\delta(\omega_1)-\delta(\omega_2)\|=2$.  Since $\delta(\omega_{\alpha}) \xrightarrow{w^\ast} \delta(\omega)$ and $\|\delta(\omega_{\alpha})\| \rightarrow \|\delta(\omega)\|$, we get a contradiction to our assumption of weak-Lim's condition. Therefore $\Omega$ is a discrete set. Thus, by \cite[Theorem 4]{LM}, we get that $C_0(\Omega)$ has Lim's condition, hence satisfies $w^*$-normal structure. Similar arguments would conclude that if $\Omega$ is compact, then the space of continuous functions $C(\Omega)$ with supremum norm has weak-Lim's condition, then $\Omega$ is finite.
\end{ex}

So, weak-Lim's condition is equivalent to Lim's condition for commutative $C^*$-algebras. For general $C^*$-algebras, we obtain that Lim's condition leads to a purely geometric characterization of a $C^*$-algebra of the form
\[
\mathcal{A}=\bigoplus_{c_0}\mathcal{L}(H_\alpha),
\]
for a family $\{H_\alpha\}_{\alpha \in \Delta}$ of finite dimensional Hilbert spaces and $\mathcal{L}(H_\alpha)$ is the space of linear and bounded operators on $H_\alpha$. Hence, we obtain a geometric classification of the class of $C^*$-algebras that are the $c_0$-direct sum of finite-dimensional operator spaces. Consequently, the dual of such $C^*$-algebras has a $w^*$-normal structure. While weak-Lim's condition characterizes $C^*$-algebras of the form
\[
\mathcal{A}=\bigoplus_{c_0}\mathcal{K}(H_\alpha),
\]
where $\mathcal{K}(H_\alpha) \subseteq \mathcal{L}(H_\alpha)$ is the space of compact operators on $H_\alpha$ (Theorem ~\ref{thmdagsubdif}). Since the norm and $w^*$-topologies coincide on the unit sphere of the dual of $\mathcal{K}(H)$ (\cite[Theorem 2]{LM}), it satisfies part $(a)$ of Lim's condition. By \cite[Theorem 3.2]{B}, we get that $\mathcal{A}$ satisfies property $(a)$ of Lim's condition.  Hence, by \cite[Lemma 1]{LM}, we get that weak-Lim's condition implies the quasi-$w^*$-normal structure on a $C^*$-algebra (see \cite{LM} for the details).

 Let $\Omega$ be a compact Hausdorff space and $C(\Omega)$ denote the space of continuous functions, equipped with the supremum norm. A uniform algebra $A \subset C(\Omega)$ is a closed subalgebra containing constants and separating points (see \cite{S}). In Proposition ~\ref{unialgebra}, as an application of the hereditary nature of weak-Lim's condition, we observe that any uniform algebra satisfying weak-Lim's condition is necessarily finite-dimensional. We use similar techniques to relate weak-Lim's condition to $k$-smoothness in the class of Banach spaces $X$ whose dual is isometric to $L^1(\mu)$ for a positive measure $\mu$. These are called $L^1$-predual or Lindenstrauss spaces and are considered as commutative analogues of $C^\ast$-algebras. Refer to \cite[Chapter 7]{L} for this discussion. As an application, we show that weak-Lim's condition implies that for any  $x \in S(X)$, $x$ is a $k$-smooth point for some $k>0$ (refer to the discussion just before Theorem ~\ref{L1predual} for the definition). 

Thus, this paper demonstrates that an apparently weaker version of Lim's condition yields structures related to the normal structure in the duals of $C^*$-algebras, certain $L^1$-predual spaces, and function spaces in a unifying way. 

This paper is organized as follows. Section 2 discusses weak-Lim's condition in $C^*$-algebras, and Section 3 is concerned with weak-Lim's condition in function spaces.

We refer to the monographs by R. B. Holmes \cite{Ho} and J. Diestel \cite{Di} for standard results from Banach space theory, \cite{L} Chapters 6 and 7, for the theory of $L^1$-predual spaces, and \cite{A1} for $C^\ast$-algebra theory. All subspaces are assumed to be closed. All bounded and compact operators are assumed to be linear.

\section{Weaker forms of Lim's condition for $C^\ast$-algebras}
In this section, we first formulate a weaker alternative to Lim's condition. For this, we use condition $(b)$ of Theorem ~\ref{thm0}. 
\begin{Def}
   $X$ is said to have property $(\dagger)$ if for any $0 <\e <2$, whenever $\{\phi_{\alpha}\}$ is a net in $B(X^*)$ with $\|\phi_{\alpha}-\phi_{\beta}\| \geq \e$ for all $\alpha \neq \beta$ and $\phi \in B(X^*)$ is such that $\phi_{\alpha} \xrightarrow{w^\ast} \phi$, then
\[
\|\phi\| \leq 1 - \e/2.
\]
\end{Def}

The arguments given during the proof of Example ~\ref{excomalg}, show that if $C_0(\Omega)$ has the property $(\dagger)$, then $\Omega$ is a discrete set.
\begin{lem} \label{daghereditary}
Let $X$ have property $(\dagger)$ and $Y$ be a subspace of $X$. Then $Y$ has property $(\dagger)$.  
\end{lem}

\begin{proof}
    Let $0 <\e <2$. Let $\{\phi_{\alpha}\}$ be a net in $B(Y^*)$ and $\phi \in B(Y^*)$ be such that $\|\phi_\alpha-\phi_\beta\| \geq \e$ for all $\alpha \neq \beta$ and $\phi_\alpha \xrightarrow{w^*} \phi$. Let $\tilde{\phi_\alpha}$ be a norm preserving extension of $\phi_\alpha$ for each $\alpha$ and $z^*$ be a $w^*$-limit of $\{\tilde{\phi_\alpha}\}$ in $B(X^*)$. 
    
    Clearly, $z^*|_{Y}=\phi$, $\|\tilde{\phi_\alpha}-\tilde{\phi_\beta\|} \geq \e$ and $\tilde{\phi_\alpha} \xrightarrow{w^*}z^*$.  So, $\|\phi\| \leq \|z^*\| <1-\e/2$. Hence, $Y$ has property $(\dagger)$.
\end{proof}

Let $X$ be an infinite dimensional reflexive Banach space such that on $S(X^\ast)$, weak and norm topologies coincide, then $X$ trivially satisfies weak-Lim's condition. The next example illustrates that such spaces can fail property $(\dagger)$.
\begin{ex} \label{exnotdag}
Let $1 <p <\infty$ and $q$ be such that $1/p + 1/q =1$. Let $\{e_n\}$ be a sequence in $\ell_q$, where $e_n$ is the sequence of elements with $1$ at the $n^{th}$ position and $0$ elsewhere. Let $\phi_n=e_n/(2)^{1/q}$ for all $n$ and $\phi=\phi_1$.
	
We have that, $\phi_n \xrightarrow{w^\ast} 0$. For all $n \geq 2$, let $\psi_n:=\phi_n + \phi$. Then, $\psi_n\xrightarrow{w^\ast} \phi$. Now, $\|\psi_n\|=1$ for all $n \geq 2$, that is, $\psi_n \in B(\ell_q)$. Also, for $n \neq m$, $\|\psi_n-\psi_m\|=1$.
	
But, $\|\phi\|=1/2^{1/q} >1 - 1/2$. Thus, $\ell_p$ does not satisfy property $(\dagger)$. Also, it is known that $\ell_p$ is embedded in $\K(\ell_p)$ and by Lemma ~\ref{daghereditary}, property $(\dagger)$ is hereditary. So, $\K(\ell_p)$ does not have property $(\dagger)$ and hence, it does not have Lim's condition.

We note that Example 4 also applies to the space  $\ell_{\infty}$ and it embeds isometrically into $\mathcal{K}(\ell_1)$ and $\mathcal{K}(\ell_\infty)$. These observations lead to an improvement of \cite[Theorem 5]{LM} for the space of compact operators $\K(\ell_2)$ to $\K(\ell_p)$ for $1 \leq p  \leq \infty$, by showing that they fail Lim's condition.
\end{ex}

The following result extends part of \cite[Theorem 4]{LM} to the case of $c_0$-sums of spaces with Lim's condition. 

\begin{thm} \label{tlimsum}
	Let $\{X_i\}_{i \in \I}$ be a collection of spaces such that $X_i$ satisfies Lim's condition for each $i \in \I$. Then, $X=\bigoplus_{c_0} X_i$ satisfies Lim's condition.
\end{thm}

\begin{proof}
	The argument follows the lines of \cite[Theorem 4]{LM}. We have, $X^*=\bigoplus_{\ell_1} X_i^*$. 
      
	Let each $X_i$ satisfy Lim's condition. Let $0 \leq s \leq 1$,  $\{\phi_\alpha\}$ be a net in $X^*$ such that $\phi_\alpha \xrightarrow{w^\ast} 0$ and $\|\phi_\alpha\|=s$ for all $\alpha$. Let $\phi \in X^*$.
    
    Clearly, $\|\phi_\alpha - \phi\| \leq s + \|\phi\|$.
	
	Let $\e >0$ be given. For any finite subset $\F$ of $\I$, $\|\phi_\alpha-\phi\| \geq \|\phi_\alpha\|-\|\phi\| + 2 \sum_{i \in \F}(|\phi(i)|-|\phi_\alpha(i)|)$. 
    
    Let $\F_0$ be a finite subset of $\I$ such that $\sum_{i \in \F_0} |\phi(i)| >\|\phi\| - \e$. Since, $\phi_{\alpha} \xrightarrow{w^\ast} 0$, there exists $\alpha_0$ such that for $\alpha \geq \alpha_0$, $\sum_{i \in \F_0} |\phi_\alpha (i)| <\e$. Thus, for $\alpha \geq \alpha_0$,
	\[
	\|\phi_\alpha - \phi\| \geq s - \|\phi\| + 2 \|\phi\| - 2\e - 2\e = s + \|\phi\| - 4\e.
	\]
	
	So, $\lim_\alpha \|\phi_\alpha - \phi\| = s + \|\phi\|$.
\end{proof}

\begin{cor}
Let $X=\bigoplus_{c_0} X_i$, where $\{X_i\}$ is a countable collection of finite-dimensional spaces. Let $Y$ be a finite-codimensional proximinal subspace of $X$. Then, $Y$ satisfies Lim's condition.
\end{cor}

\begin{proof}
Since $Y$ is a proximinal subspace of finite codimension, from the proof of \cite[Proposition 3.10 ((ii) $\implies$ (i))]{NR}, we get the decomposition, $Y=Y_1 \bigoplus_{\infty}Y_2$, where $Y_1$ is finite-dimensional and $Y_2=\sum_{i \in \N \setminus I}\bigoplus_{c_0}X_i$, where $I$ is a finite set. By Theorem ~\ref{tlimsum}, $Y_2$ satisfies Lim's condition. $Y_1$, being finite-dimensional, also satisfies Lim's condition. Hence, $Y$ satisfies Lim's condition.
\end{proof}
As a consequence of Theorem ~\ref{tlimsum}, we obtain a geometric classification of $C^*$-algebras which are $c_0$-direct sums of finite-dimensional operator spaces.

\begin{thm} \label{thmlimchar}
Let ${\mathcal A}$ be a $C^\ast$-algebra. ${\mathcal A}$ has Lim's condition if and only if it is isometrically $*$-isomorphic to a $c_0$-direct sum $\bigoplus_{c_0} {\mathcal L}(H_{\alpha})$, for a family of finite-dimensional Hilbert spaces $\{H_{\alpha}\}_{\alpha \in J}$, for some index set $J$, and spaces of bounded operators ${\mathcal L}(H_{\alpha})$.
\end{thm}

\begin{proof}
Let $\mathcal A$ have Lim's condition. So, $\mathcal{A}$ has property ($\dagger$). We verify Condition $(iii)$ of \cite[Theorem 2]{CPW}, to get the required conclusion. Let $x \in S({\mathcal A})$  be a normal element and let $\sigma(x)$ denote the spectrum of $x$. Consider the commutative $C^*$-algebra generated by $\{x,~x^*\}$. By the Gelfand-Naimark theorem, we have the identification of this separable subalgebra as $C_0(\sigma(x))$.

We also have that $\sigma(x)$ is a locally compact metric space. Since property ($\dagger$) condition is a hereditary property, the separable commutative $C^*$-subalgebra $C_0(\sigma(x))$ has property ($\dagger$). Then, following the argument in Example ~\ref{excomalg}, we get that $\sigma(x)$ is a discrete set with $0$ as the only possible accumulation point. Hence, we get that $\mathcal{A}$ is isometric and $*$-isomorphic to, a decomposition $\bigoplus_{c_0} {\mathcal K}(H_{\alpha})$, for a family of Hilbert spaces $\{H_{\alpha}\}_{\alpha \in J}$, for some index set $J$, and spaces of compact operators ${\mathcal K}(H_{\alpha})$, from \cite[Theorem 2]{CPW}. Now, by the hereditary nature of property $(\dagger)$ and Example ~\ref{exnotdag}, we see that the Hilbert spaces under consideration are finite-dimensional.

\enlargethispage{\baselineskip}
\enlargethispage{\baselineskip}

The converse follows from Theorem ~\ref{tlimsum}.
	
\end{proof}

As a consequence of the results in \cite[Lemma 4]{LM} and \cite[Theorem 1]{Lim}, we have
\begin{cor}
$\ell_1$-direct sum of duals of finite dimensional operator spaces has $w^*$-normal structure, hence, every non-expansive map on $w^*$-compact sets has a fixed point. 
\end{cor}

Now we look at weak-Lim's condition in $C^*$-algebras and uniform algebras.

\begin{prop} \label{w-limhereditary}
	If $X$ has weak-Lim's condition, then every closed subspace of $X$ has weak-Lim's condition.
\end{prop}

\begin{proof}
	Let $Y$ be a closed subspace of $X$ and $0 <\e <2$. Consider $\{\phi_\alpha\}$ to be a net in $B(Y^*)$ and $\phi \in B(Y^*)$ such that $\|\phi_\alpha\| \rightarrow \|\phi\|$, $\phi_\alpha \xrightarrow{w^*} \phi$ and $\|\phi_\alpha - \phi_\beta\| \geq \e$ for all $\alpha \neq \beta$. Let $\ol{\phi_\alpha}$ be any extension that preserves the norm of $\phi_\alpha$ and assume without loss of generality that $\ol{\phi_\alpha} \xrightarrow{w^*} z^*$ for some $z^* \in B(X^*)$. Clearly, $z^*|_Y=\phi$ and $\|z^*\| \geq \|\phi\|
    $. By the weak$^*$ lower-semi-continuity of the norm, $\|z^*\| \leq \liminf\|\ol{\phi_\alpha}\|= \liminf\|\phi_\alpha\|$. Since, $\|\phi_\alpha\| \rightarrow \|\phi\|$, $\|\ol{\phi_\alpha}\| \rightarrow \|\phi\|$ and $\|z^*\|=\|\phi\|$. Therefore, $z^*$ is a norm attaining functional.
	
	Notice that, we have $\| \ol{\phi_\alpha}-\ol{\phi_\beta}\| \geq \|\phi_\alpha - \phi_\beta\| \geq \e$. So, by weak-Lim's condition on $X$, $\|z^*\| \leq 1-\e/2$.
	
	So, $\|\phi\| \leq \|z^*\| \leq 1-\e/2$.
\end{proof}

Let $\Omega$ be a compact Hausdorff space. Another application of Proposition ~\ref{w-limhereditary} is a version of Example ~\ref{excomalg} for subalgebras. We recall that a closed subalgebra $A \subset C(\Omega)$ is said to be a uniform algebra if $1 \in A$ and $A$ separates points of $\Omega$ (see \cite{S}). By the classical Stone-Weierstrass theorem, when $A$ is a proper subalgebra, it is not self-adjoint.
\begin{prop} \label{unialgebra}
    Let $A$ be a uniform algebra having weak-Lim's condition. Then $A$ is isometric to the finite dimensional space $\ell^\infty(k)$ for a positive integer $k$.
\end{prop}
\begin{proof} A general theorem proved in \cite{Rao} (Theorem 1) says that if $A$ has a Banach space theoretic property that is hereditary and any $C(\Omega)$ space having that property is finite dimensional, then $A$ itself is finite dimensional. From Proposition ~\ref{w-limhereditary} and Example ~\ref{excomalg}, we see that weak-Lim's condition is such a property. Since $A$ separates points of $\Omega$, we get that $\Omega$ is a finite set and $A$ is isometric to $\ell^\infty(k)$ for a positive integer $k$.
    
\end{proof}
We need a technical lemma to further investigate the weak-Lim's condition.

\begin{lem} \label{lnormal}
	Let $X$ have weak-Lim's condition and $0 \leq s \leq 1$. Then, for every $0 <\e <2s$, whenever $\{\phi_{\alpha}\}$ is a net of norm attaining functionals in $X^*$ with $\|\phi_{\alpha}\| \leq s$ for all $\alpha$ and $\phi \in B(X^*)$ is such that $\phi$ is norm attaining, $\phi_{\alpha} \xrightarrow{w^\ast} \phi, \|\phi_\alpha\| \rightarrow \|\phi\|$ and $\|\phi_{\alpha}-\phi_{\beta}\| \geq \e$ for all $\alpha \neq \beta$, then
	\[
	\|\phi\| \leq s - \e/2.
	\]
\end{lem}

\begin{proof}
	Let $0 \leq s \leq 1$ and $0 <\e <2s$. Suppose $\{\phi_{\alpha}\}$ is a net of norm attaining functionals in $X^*$ with $\|\phi_{\alpha}\| \leq s$ for all $\alpha$ and $\phi \in B(X^*)$ is such that $\phi$ is norm attaining, $\phi_{\alpha} \xrightarrow{w^\ast} \phi, \|\phi_\alpha\| \rightarrow \|\phi\|$ and $\|\phi_{\alpha}-\phi_{\beta}\| \geq \e$ for all $\alpha \neq \beta$.
    
	We have, $0 <\e/s <2$. Let $\psi_{\alpha}=\phi_{\alpha}/s$ for all $\alpha$ and $\psi=\phi/s$. Then, $\psi_\alpha \in B(X^*)$ for all $\alpha$, $\psi_{\alpha} \xrightarrow{w^\ast} \psi, \|\psi_\alpha\| \rightarrow \|\psi\|$ and $\|\psi_{\alpha}-\psi_{\beta}\| \geq \e/s$. Therefore,
	$\|\psi\| \leq 1 - \e/2s$. This implies, $\|\phi\| \leq s - \e/2$.
\end{proof}

We next consider the stability of weak-Lim's condition under $\ell_\infty$-sum.
\begin{thm} \label{tstab}
	Let $Z = X \bigoplus_{\ell_\infty} Y$ for some Banach spaces $X$ and $Y$. Then $Z$ has weak-Lim's condition if and only if $X$ and $Y$ do.
\end{thm}

\begin{proof}
	Clearly, if $Z$ satisfies weak-Lim's condition, then, since $X$ and $Y$ are subspaces of $Z$ and weak-Lim's condition is hereditary, $X$ and $Y$ also satisfy weak-Lim's condition.
	
	For the proof of the converse part, we first note that $Z^* = X^* \bigoplus_{\ell_1} Y^*$.
	We denote the norms on $X$ and $Y$ also by $\|\cdot\|$.
    
	Let $0 <\e <2$, $\{\phi_{\alpha}\}$ be a net of norm attaining functional in $B(Z^*)$ and $\phi \in B(Z^*)$ be norm attaining such that $\|\phi_\alpha\| \rightarrow \|\phi\|$, $\phi_{\alpha} \xrightarrow{w^\ast} \phi$ and $\|\phi_{\alpha}-\phi_{\beta}\| \geq \e$ for all $\alpha \neq \beta$. We now choose components from the corresponding summands in the dual and also assume, without loss of generality, that the components of $\phi$ are non-zero. 
    
    For each $\alpha$, let $x^*,y^*,x_{\alpha}^*, y_{\alpha}^*$ be such that $\phi_{\alpha}=x_{\alpha}^* + y_{\alpha}^*$ for all $\alpha$ and $\phi=x^* + y^*$. We have, $\|\phi_\alpha\| = \|x_{\alpha}^*\|+\|y_{\alpha}^*\| \leq 1$. Passing to subnets, if necessary, we may assume that $\|x_{\alpha}^*\|=s$ for some $0 \leq s \leq 1$ and for all $\alpha$. So, for each $\alpha$, $\|y_\alpha^*\| \leq 1-s$. We have, $\|x_\alpha^*\| \rightarrow \|x^*\|=s$ and $\|y^*_\alpha\| \rightarrow \|y^*\|$. 
    We also have, $x_\alpha^* \xrightarrow{w^\ast} x^*$ and $y_\alpha^* \xrightarrow{w^*} y^*$.
	
	Also, $\|\phi_\alpha - \phi_\beta\|=\|x_\alpha^*-x_\beta^*\| + \|y_\alpha^*-y^*\| \geq \e$. Again, passing to subnets if necessary, let $s'=\|x_\alpha^*-x_\beta^*\|$ for all $\alpha \neq \beta$. So, $\|y_\alpha^*-y_\beta^*\| \geq \e-s'$ for all $\alpha \neq \beta$.
	
	Since, $X$ and $Y$ have weak-Lim's condition, by Lemma ~\ref{lnormal},
	$\|x^*\| \leq s - s'/2$ and $\|y^*\| \leq (1-s) - (\e-s')/2$. Thus,
	\[
	\|\phi\|=\|x^*\| + \|y^*\| \leq 1 - \e/2.
	\]
\end{proof}

Clearly, the above argument extends to finite sums.

Now, we look at the stability of weak-Lim's condition for $c_0$-sums of Banach spaces.

Let $I$ denote the indexing set for a family of Banach spaces. For notational convenience, we omit writing the indexing set.
\begin{thm} \label{thmwlimstability}
	Let $\{X_\gamma\}$ be any collection of  Banach spaces. Then $X=\bigoplus_{c_0}X_\gamma$ has weak-Lim's condition if and only if $X_\gamma$ has weak-Lim's condition for each $\gamma$.
\end{thm}

\begin{proof}
Again, the hereditary nature of weak-Lim's condition implies that $X_\gamma$ has weak-Lim's condition for each $\gamma$.
	
Conversely, we have $X^*=\bigoplus_{\ell_1}X_\gamma^*$.
Again, we denote the norm on each $X_\gamma$ as $\|\cdot\|$. In what follows, we use the standard identification of $X^\ast$, in particular, when the indexing set $I$ is uncountable, $0 \neq \phi \in B(X^\ast)$ has only countably many non-zero coordinates.
	
Let $0 <\e <2$, $\{\phi_\alpha\}$ be a net of norm-attaining functionals in $B(X^*)$ with $\|\phi_\alpha - \phi_\beta\| \geq \e$ for all $\alpha \neq \beta$ and $\phi \in B(X^*)$ be a norm-attaining functional such that $\phi_{\alpha} \xrightarrow{w^\ast} \phi$ and $\|\phi_\alpha\|\rightarrow \|\phi\|$.
\vskip .5em
We first note that any norm-attaining $0 \neq \phi \in B(X^*)$ has only finitely many non-zero coordinates. To see this, suppose $x \in S(X)$ and $$\|\phi\|= \sum \phi(\gamma)(x(\gamma))\leq \sum \|\phi(\gamma)\| \|x(\gamma)\|\leq \|\phi\|.$$

Thus, whenever $\phi(\gamma) \neq 0$, $\|x(\gamma)\|=1$ and it is easy to see that the corresponding $\phi(\gamma)$ attains its norm at $x(\gamma) \in S(X_{\gamma})$. Since $\{\|x(\gamma)\|\}$ vanishes at infinity, we obtain that $\phi$ has only finitely many non-zero coordinates.
	
	 So, without loss of generality, we may assume that for each $\alpha$, $\phi_{\alpha}$ has only first $n$ coordinates non-zero.
	
	Therefore, proceeding as in the proof of  Theorem ~\ref{tstab}, we obtain that $\|\phi\| \leq 1-\e/2$.
\end{proof}

Now, we obtain a geometric classification of $C^*$-algebras satisfying weak-Lim's condition. As mentioned in the introduction, this gives us that $C^*$-algebras satisfying weak-Lim's condition have quasi-$w^*$-normal structure.

\begin{thm} \label{thmdagsubdif}
	
	 Let ${\mathcal A}$ be a $C^\ast$-algebra. ${\mathcal A}$ satisfies weak-Lim's condition if and only if ${\mathcal A}$ is isometrically $*$-isomorphic to  a $c_0$-direct sum $\bigoplus_{c_0} {\mathcal K}(H_{\alpha})$, for a family of Hilbert spaces $\{H_{\alpha}\}_{\alpha \in J}$, for some index set $J$, and spaces of compact operators ${\mathcal K}(H_{\alpha})$.
\end{thm}

\begin{proof}
	 Let ${\mathcal A}$ satisfy weak-Lim's condition. As done in Theorem ~\ref{thmlimchar}, we again verify condition $(iii)$ of \cite[Theorem 2]{CPW}. Let $x \in S({\mathcal A})$  be a normal element and consider the commutative $C^*$-algebra generated by $\{x,~x^*\}$. By the Gelfand-Naimark theorem, we have the identification of this separable subalgebra as $C_0(\sigma(x))$. We also have that $\sigma(x)$ is a locally compact metric space. Since weak-Lim's condition is a hereditary property, the separable commutative $C^*$-subalgebra $C_0(\sigma(x))$ has weak-Lim's condition.  By Example 4 again, we get that $\sigma(x)$ is a discrete set with $0$ as the only possible accumulation point. Hence, we have the required decomposition from \cite[Theorem 2]{CPW}.

     Conversely, let $\{H_\alpha\}$ be any collection of Hilbert spaces. We claim that $\mathcal{K}(H)$ satisfies weak-Lim's condition for any Hilbert space $H$. Let $\{\phi_\gamma\}$ be a net in $B(\mathcal{K}(H)^*)$ such that $\phi_\gamma \xrightarrow{w^*} \phi$ for some $\phi \in B(\mathcal{K}(H)^*)$ and $\|\phi_\gamma\| \rightarrow \|\phi\|$. By the proof of \cite[Theorem 2]{LM}, $\phi_\gamma \rightarrow \phi$ in the norm. Thus, $\mathcal{K}(H)$ satisfies weak-Lim's condition. 
     
     So, by Theorem ~\ref{thmwlimstability}, $\bigoplus_{c_0} {\mathcal K}(H_{\alpha})$ satisfies weak-Lim's condition.
	
\end{proof}

The following is an interesting consequence of the hereditary nature of weak-Lim's condition. We observe that 

\begin{cor}Let ${\mathcal A}$ be a $C^\ast$-algebra. If every separable commutative $C^\ast$-subalgebra has weak-Lim's condition, then ${\mathcal A}$ has weak-Lim's condition.
\end{cor}

\begin{proof} 
Suppose every separable commutative $C^*$-subalgebra has weak Lim's condition. Then, as before, considering the spectrum of any normal element $a$, since $C_0(\sigma(x))$ is a separable commutative subalgebra, it has the weak-Lim's condition. Thus $\sigma(a)$ is a discrete set with $0$ as the only possible accumulation point. Hence, we have the $c_0$-decomposition of ${\mathcal A}$ into the space of compact operators on Hilbert spaces. Thus, ${\mathcal A}$ has weak-Lim's condition by Theorem ~\ref{thmdagsubdif}.
\end{proof}

We recall that $Y \subset X$ is said to be a proximinal subspace if for all $x \in X$, there is a $y \in Y$ such that  $d(x,Y)= \|x-y\|$. In particular, if $P$ is a contractive projection on $X$, then $ker(P)$ is a proximinal subspace as $d(x,ker(P))=\|P(x)\|=\|x-(x-P(x))\|$. Also, if $Y\subset X$ is a subspace of finite codimension and proximinal, then by Gargavi's theorem, $Y$ is a finite intersection of kernels of norm-attaining functionals. (See \cite{Sin} and \cite[Lemma 2.1]{NR}).

We do not know if Lim's condition is hereditary in general. But under some assumptions, we obtain the hereditary nature of Lim's condition.

\begin{rem}
	Let $X$ satisfy Lim's condition and $Y$ be a closed subspace of $X$. Then $X/Y$ has Lim's condition. To see this, we recall that $(X/Y)^*$ is canonically identified with $Y^\perp$. Let $0 \leq s \leq 1$ and $\{\phi_\alpha\}$ be a net in $Y^\perp$ with $\|\phi_\alpha\|=s$ and $\phi_\alpha \xrightarrow{w^\ast} 0$. Let $\phi \in Y^\perp$. Since, $X$ has Lim's condition, $\lim_\alpha \|\phi_\alpha + \phi\|=s+\|\phi\|$.
\end{rem}

To preserve weak-Lims' condition on the quotient spaces, we need an additional hypothesis of proximality.

\begin{prop}
	Suppose $X$ has weak-Lim's condition and $Y\subset X$ is a proximinal subspace. Then $X/Y$ has weak-Lim's condition.	
\end{prop}

\begin{proof}
	We have that $Y^\perp \cong (X/Y)^*$ via the canonical isometry. Let $0 <\e <2$, $\{\phi_\alpha\}$ be a net of norm-attaining functionals in $B((X/Y)^*)$ such that $\|\phi_\alpha - \phi_\beta\| \geq \e$ and $\phi \in B((X/Y)^*)$ be a norm-attaining functional with $\|\phi_\alpha\| \rightarrow \|\phi\|$ and $\phi_\alpha \xrightarrow{w^\ast} \phi$. We can identify $\{\phi_\alpha\}$ and $\phi$ as elements in $Y^\perp$.
	
	Now, $\phi$ attains norm on $B(X/Y)$. So, there exists $x_0 \in X$ such that $\phi(x_0)=\|\phi\|=d(x_0, Y)$. But, $Y$ is proximinal, thus, there is $y_0$ in $Y$ such that $\phi(x_0)=\phi(x_0-y_0)=\|\phi\|=\|x_0-y_0\|$. So, $\phi$ is norm-attaining as an element of $Y^\perp$. Similarly, for each $\alpha$, $\phi_\alpha$ is norm-attaining as an element of $Y^\perp$.
	
	Now, weak-Lim's condition is hereditary. So, $Y^\perp$ has weak-Lim's condition. Thus, $(X/Y)$ has weak-Lim's condition.
\end{proof}

An interesting observation relates reflexivity and Lim's condition in a dual space. If $X^\ast$ has Lim's condition, then weak$^\ast$ and norm topologies coincide on $S(X^{\ast\ast})$. By the Goldstein's theorem (see \cite{Ho}), under the canonical embedding, $B(X)$ is $w^*$-dense in $B(X^{\ast\ast})$. So, we see that Lim's condition in a dual space implies reflexivity. We do not know in general if $X^\ast$ satisfies weak-Lim's condition, then $X$ is reflexive. We have a partial result when $X$ has no isomorphic copy of $\ell_1$ ($\ell_1 \not\hookrightarrow X$). We will be using standard properties of reflexive Banach spaces. See \cite{Ho}.

  \begin{thm} \label{thmreflexivity}
  	If $X$ is such that $\ell_1 \not\hookrightarrow X$ and $X^*$ has weak-Lim's condition, then $X$ is reflexive, and the norm and weak topologies coincide on $S(X)$.
  \end{thm}
  
  \begin{proof}
  	Let $Y \subset X$ be a separable subspace. Now $Y$ has weak-Lim's condition and  $\ell_1 \not\hookrightarrow Y$.  By a result in \cite[Pg. 215]{Di}, every element of $S(Y^{**})$ is a $w^*$-sequential limit of elements of $S(Y)$. Let $\phi \in S(Y^{**})$ be a norm-attaining functional and $\{y_n\}$ be a sequence in $S(Y)$ such that $y_n \xrightarrow{w^\ast} \phi$. If $\{y_n\}$ does not converge to $\phi$ in the norm, then  $\{y_n\}$ is not Cauchy. Then there is a subsequence, again denoted by $\{y_n\}$ such that for some $\e >0$, this subsequence is $\e$-separated, i.e.,  $\|y_n-y_m\| \geq \e$ for all $n \neq m$. Clearly $y_n$'s are norm-attaining on $Y^\ast$, $\|y_n\|=\|\phi\|=1$ for all $n$ and $y_n \rightarrow \phi$ in the $w^*$-topology. Since $Y$ has weak-Lim's condition, we have $\|\phi\| \leq 1-\e/2$, which is a contradiction. So, we get that $\{y_n\}$ is a Cauchy sequence in $Y$ and hence $\phi=y$ for some $y\in Y$. Thus, all norm-attaining functionals in $S(Y^*)$ belong to $Y$. Therefore, by the Bishop-Phelps theorem on the norm-denseness of norm attaining functional in $Y^*$ (see \cite{Ho}), we get that $Y = Y^{**}$. Thus, $Y$ is reflexive. Since every separable subspace of $X$ is reflexive, we get that $X$ is reflexive. 

    As $X$ is reflexive, sequential convergences determine the weak topology on $S(X)$. Thus, as in the first part of the proof, since $X^*$ has weak-Lim's condition, we see that weak and norm topologies coincide on $S(X)$.  		
  	
  \end{proof}

  \begin{cor} \label{cornormal}
If $\ell_1 \not \hookrightarrow X$ and $X^*$ satisfies weak-Lim's condition, then every weakly compact convex subset $K$ of $X$ has a normal structure and hence, every non-expansive map on $K$ has a fixed point.
  \end{cor}

  \begin{proof}
    By \cite[Proposition 3.2.1]{vDS} and Theorem ~\ref{thmreflexivity}, $X$ has $\e$-UKK (we refer to Section 3 of \cite{vDS} for the definition and related notions). So, by \cite[Theorem 3.2.2]{vDS}, $X$ has $w$-normal structure, and thus, every non-expansive map on a weakly compact set has a fixed point.
  \end{proof}

\section{Weak-Lim's condition on function spaces}

In this section, we study weak-Lim's condition on $L_1$-predual spaces and $A(K)$, the space of affine continuous functions on a compact set $K$. 

For a compact convex set $K$, let $\partial_e K$ denote the set of extreme points of $K$.

The following lemmas are crucial for the upcoming discussion.

\begin{lem} \label{lemsubalg}
Let $X \subset C(\Omega)$ be a closed subspace and suppose ${\bf 1} \in X$. Then $\overline{\partial_e B(X^\ast)}^{w^*} \subset S(X^\ast)$ and all the functionals in $\overline{\partial_e B(X^\ast)}^{w^*} $ are norm-attaining.
\end{lem}

\begin{proof}
By an application of the Krein-Milman theorem, any element $x^*$ of $\partial_e B(X^\ast)$ can be extended to an element of $\partial_e C(\Omega)^\ast$. So, $|x^*({\bf 1})|=1$. Now, the $w^*$-limit of such a point also takes the absolute value $1$ at the constant function ${\bf 1}$. Thus,  $\overline{\partial_e B(X^\ast)}^{w^*} \subset S(X^\ast)$ and all the functionals in the $w^*$-closure attain their norm at ${\bf 1}$.
\end{proof}

The following lemma appears as \cite[Theorem 2]{IL}. We prove it for the sake of completeness.

\begin{lem} \label{lemonedsum}
	Let $x^* \in \partial_e B(X^*)$ be a one dimensional $L$-summand, that is, $X^*=span\{x^*\} \bigoplus_1 N$ for some closed subspace $N$ of $X^*$. Then, $x^*$ attains its norm.
\end{lem}

\begin{proof}
	Since $X^\ast = span\{x^\ast\} \bigoplus_1 N$, we note that if $y^\ast \in S(X^\ast)$ is norm-attaining and $y^\ast = \alpha x^\ast+n^\ast$ for some $n^\ast \in N$, with both components non-zero, then it is easy to see that $x^\ast$ also attains its norm. Now, using the Bishop-Phelps theorem on the denseness of norm-attaining functionals (\cite[Pg. 169]{Ho}), we get that there is a norm-attaining functional $y^\ast$ for which the components in the decomposition are non-zero. Thus, $x^\ast$ attains its norm.
\end{proof}

  \vskip 1em
   We recall that if $X$ is an $L^1$-predual space, then since $X^\ast = L^1(\mu)$, for any $x^\ast \in \partial_e B(X^*)$ we have $X^\ast = span\{x^\ast\} \bigoplus_1 N$. To see this, we recall that any extreme point of $B(L^1(\mu))$ is determined, up to a scalar multiple, by the indicator function of a $\mu$-atom.  So, by Lemma ~\ref{lemonedsum}, $x^*$ attains norm. This is the motivation for studying weak-Lim's condition in this context.
  \vskip 1em

  Let $A(K)$ denote the space of complex-valued, affine continuous functions on a compact set $K$, equipped with the supremum norm. In what follows, we will use some of the terminology and results from \cite{AE}. See also Section 18 of \cite{L}.
 Let $\phi:K \rightarrow A(K)^*$ denote the evaluation map. It is known that $\partial_e B(A(K)^*)=\{t\phi(k): k\in \partial_e K~,~|t|=1\}$. For an extreme point $k$ of $K$, that is also a split face, it is known that $span(\delta(k))$ is a one dimensional $L$-summand in $A(K)^*$ (see \cite[Chapter 3, Theorem 8.3]{AE}). This in particular happens when $K$ is a Choquet Simplex. In this case, to connect it with the $C^*$ algebra theory, it follows from \cite[Theorem 1.1]{BR} that the set of tracial states $T(\mathcal{A})$ for a unital $C^*$-algebra $\mathcal{A}$ is a Choquet simplex. A compact convex set $K$ is a Choquet simplex if and only if $A(K)$ is an $L^1$-predual space.
\vskip 1em

\begin{thm} \label{thmunialg}
Let $K$ be a Choquet simplex. If $A(K)$ satisfies weak-Lim's condition, then $K$ is a finite dimensional simplex.
\end{thm}

\begin{proof} 
Suppose $A(K)$ has weak-Lim's condition. Note that the space $A(K)$ satisfies the hypothesis of Lemma 23. Let $k \in \overline{\partial_e K}$. Suppose a net $\{k_{\alpha}\} \subset \partial_e K$ be a non-constant net and $k_{\alpha} \rightarrow k$. Then $\phi(k_{\alpha}) \rightarrow \phi(k)$ in the $w^*$-topology, and all functionals are unit vectors attaining their norm. 
By our discussion in the preamble, for $\alpha \neq \beta$ in $\Delta$, $\|\phi(k_{\alpha})-\phi(k_{\beta})\|=2$. Since $A(K)$ has weak-Lim's condition, this can happen only when $\partial_eK$ is a finite set. Thus, $K$ is a finite dimensional simplex.

\end{proof}

We next give an application of the above result.

Recall from \cite{LR} that a non-zero element $x \in X$ is said to be a \textit{k-smooth} point if the set of supporting functionals $J(x):=\{f \in S(X^*): f(x)=\|x\|\}$ is $k$-dimensional, that is, $dim(span(J(x)))=k$.

\begin{thm} \label{L1predual}
Let $X$ be an $L^1$-predual space satisfying weak-Lim's condition. Then, for any $x \in S(X)$, $x$ is a $k$-smooth point for some $k>0$.
 \end{thm}
 
 \begin{proof} Let $x_0 \in S(X)$. Denote by $S_{x_0}:=\{f \in S(X^*) : f(x_0)=1\}$. Since $X$ is an $L^1$-predual space and as $S_{x_0}$ is a w$^*$-closed extreme convex subset (face) of $B(X^*)$,
 it follows from \cite[Theorem 3.3]{E}  that $span\{S_{x_0}\}$ is a w$^*$-closed subspace and $X^\ast = span\{S_{x_0}\} \bigoplus_1 N$ for a closed subspace $N \subset X^\ast$. 
 
 Now, if $J = \{x \in X: x(S_{x_0})=0\}$, then $J^\bot = span\{S_{x_0}\}$ and is a proximinal subspace of $X$ (see \cite{HWW}, Chapter I). Let $\pi: X \rightarrow X/J$ be the quotient map and  $\Phi: X/J \rightarrow A(S_{x_0})$ be defined as $\Phi(\pi(x))(x^\ast)= x^\ast(x_0)$, where $x^\ast \in S_{x_0}$ and $x \in X$. This is a linear map taking $\pi(x_0)$ to $1$ and it is easy to see that $\Phi$ is a surjective isometry. 
 
 It now follows from \cite[Theorem 2, Section 19]{L} and its complex versions from \cite[Section 20]{L}, that $S_{x_0}$ is a Choquet simplex. Since $X/J$ satisfies weak-Lim's condition, we obtain from Theorem~\ref{thmunialg} that $X/J$ is isometric to $\ell^\infty(k)$. Consequently, $x_0$ is a $k$-smooth point.
 \end{proof}
 
  \begin{rem} It is reasonable to conjecture that if $X$ is an $L^1$-predual space and satisfies weak-Lim's condition, then it is isometric to $c_0(\Gamma)$ for a discrete set $\Gamma$.
\end{rem}
We have a partial positive result for a special class of complex $L^1$-predual spaces, called $G$-spaces, using the property $ (\dagger) $. We use the definition and the geometric classification of these spaces from \cite{O}.
\begin{Def}
    Let $\Omega$ be a compact Hausdorff space and let $T$ denote the unit circle. Let $\{(\omega_{\alpha},\omega'_{\alpha}, t_{\alpha}, \lambda_{\alpha})\}_{\alpha \in \Delta} \subset \Omega \times \Omega \times T \times[0,1]$ be a collection of points in the product space. The function space, $G= \{f \in C(\Omega): f(\omega_{\alpha})= \lambda_{\alpha}t_{\alpha}f(\omega'_{\alpha})~for~all~\alpha \in \Delta\}$ is called a $G$-space. 
\end{Def}
Theorem 24 of \cite{O} states that, a complex $L^1$-predual space is isometric to a $G$-space if and only if $\overline{\partial_eB(X^*)} ^{w^*}\subset [0,1]\partial_eB(X^*)$.
\begin{thm}
    Let $X$ be a $G$-space having the property $(\dagger)$. Then $X$ is isometric to $c_0(\Gamma)$ for a discrete set $\Gamma$. 
\end{thm}
\begin{proof}
    Let $\phi \in \overline{\partial_eB(X^*)}^{w^*}$. By hypothesis and Theorem 24 of \cite{O}, $\phi= \lambda x^*$ for some $x^*\in \partial_e B(X^*)$, $\lambda \in [0,1]$. Suppose $\lambda \in (0,1]$, then $\|\phi\|=\lambda$. Let $\{x^*_{\alpha}\} \subset \partial_e B(X^*)$ be a net such that $x^*_{\alpha} \xrightarrow{w^*} \phi$. Since $X$ is an $L^1$-predual space, for $\alpha \neq \beta$, $\|x^*_{\alpha}-x^*_{\beta}\|=2$. Thus, for $\e >0$ with $1-\frac{\e}{2} < \lambda$ we get a contradiction. Similar arguments show that $\partial_eB(X^*)$ is a discrete set with $0$ as the only possible $w^*$-accumulation point. Therefore $X$ is isometric to $c_0(\Gamma)$ for a discrete set $\Gamma$ (see Theorem 12 in \cite{O}).
\end{proof}
 \begin{rem}
In \cite{CPW}, the authors characterize Kaplansky's $I$-rings among $C^\ast$-algebras in terms of the denseness of points of strong subdifferentiability in the unit sphere. It would be interesting	to get an analogue of this in terms of variations on Lim's conditions or fixed point property that we have considered here.  Classification schemes related to these ideas will appear elsewhere.
 \end{rem}
\smallskip
Some of these results are part of an ongoing research project of the second author, funded by the Anusandhan National Research Foundation (ANRF) Core Research Grant 2024-2027, CRG-2023/ No-000595, 'Classification of Banach spaces using differentiability'.


\smallskip
\begin{thebibliography}{99}
\bibitem{A1} W. B. Arveson, {\em An invitation to $C^*$-algebras}, Springer GTM 39, Berlin 1976.

\bibitem{AE} L. Asimow and A.~J. Ellis, {\em Convexity theory and its applications in functional analysis}, London Mathematical Society Monographs, 16, Academic Press, London-New York, 1980.

\bibitem{B} S. Basu and T. S. S. R. K. Rao, {\em Some stability results for asymptotic norming properties in Banach spaces}, Colloq. Math. 75 (1998) 271--284.

\bibitem{BR} B. Blackadar and M. R{\o}rdam, {\em The space of tracial states on a $C^*$-algebra},  Expositiones Mathematicae (2024) 125618,
https://doi.org/10.1016/j.exmath.2024.125618. 

\bibitem{CPW} M. D. Contreras, R. Pay\'a and W. Werner, {\em $C^*$-algebras that are I-rings}, J. Math. Anal. Appl. 198 (1996) 227--236.

\bibitem{Di} J. Diestel, {\em Sequences and series in Banach spaces}, Graduate Texts in Mathematics, 92, Springer, New York, 1984.

\bibitem{vDS} D. van~Dulst and B. Sims, {\em Fixed points of nonexpansive mappings and Chebyshev centers in Banach spaces with norms of type (KK)}, Banach space theory and its applications (Bucharest, 1981), 35--43, Lecture Notes in Math. Lecture Notes in Math., 991 991, Springer, Berlin-New York.

\bibitem{E} A. J. Ellis, A. K. Roy, T. S. S. R. K. Rao and U. Uttersrud, {\em Facial structure of complex Lindenstrauss spaces}, Trans. Amer. Math. Soc., 268 (1981) 173--186.
    

    
\bibitem{HWW} P. Harmand, D. Werner and W. Werner, {\em $M$-ideals in Banach spaces and Banach algebras}, Springer LNM 1547, Berlin 1993.

\bibitem{Ho} R. B. Holmes, {\em Geometric Functional analysis and its applications}, Springer. GTM 24, Berlin 1975.

\bibitem{IL} V. Indumathi and S. Lalithambigai, {\em A new proof of proximinality for $M$-ideals}, Proc. Amer. Math. Soc. 135 (2007) 1159--1162.

\bibitem{L} H. E.  Lacey, {\em  The isometric theory of classical Banach spaces}, In: DieGrundlehren der Mathematischen Wissenschaften, Band, vol. 208, pp. x+270. Springer, New York (1974).

\bibitem{Lim} T. C. Lim, {\em Asymptotic centers and nonexpansive mappings in conjugate spaces}, Pacific J. Math., 90 (1980) 135--143.

\bibitem{LM} A. T. M. Lau and P. F. Mah, {\em Quasi-normal structure for certain spaces of operators on Hilbert spaces}. Pacific J. Math. 121 (1986) 109--118.

\bibitem{LR} B. L. Lin and T. S. S. R. K. Rao, {\em Multi smoothness in Banach spaces}, Int. J. Math. Math. sci., Art. ID 52382, (2007), 12pp.

\bibitem{NR} D. Narayana and T.~S.~S.~R.~K. Rao, {\em Transitivity of proximinality and norm-attaining functionals}, Colloq. Math., 104 (2006),
 1--19.
\bibitem{O} G. H. Olsen, {\em Classification of complex Lindenstrauss spaces}, Math. Scand. 35 (1974) 237-258.
\bibitem{Rao} T.S.S.R.K. Rao, \emph{Uniform algebras as Banach spaces}, Complex Anal. Oper. Theory, (2026), article 48.
 
 \bibitem{Sin} I. Singer, {\em Best approximations in normed linear spaces by elements of linear subspaces}, Springer, Berlin, 1970.
 \bibitem{S} E. L. Stout, {\em The theory of Uniform algebras}, Bogden and Quigly, New York 1971.
 \end{thebibliography}
\end{document}